\documentclass[12pt]{article}

\usepackage{amsfonts,mathrsfs}
\usepackage{amssymb,amsmath,amsbsy,amsthm}
\usepackage{dsfont}
\usepackage{blkarray}
\usepackage{tikz}
\usepackage{tkz-euclide}
\usepackage{tikz-cd}
\usetikzlibrary{patterns}
\usepackage{ae}
\usepackage{bm}
\usepackage{graphicx}
\usepackage{float}
\usepackage{cite}
 \usepackage{indentfirst}
\usepackage{lineno}
\usepackage{titlesec}
\usepackage{enumitem}
\usepackage{color}
\usepackage{aliascnt}
\usepackage{varioref}
\usepackage{hyperref}
\hypersetup{colorlinks=true,linkcolor=cyan,anchorcolor=cyan,citecolor=red,CJKbookmarks=True,
,urlcolor=cyan}
\usepackage{cleveref}

\numberwithin{equation}{section}

\def\int{\mbox{\rm int}}

\def\And{\mbox{\rm ~and~}}

\def\For{\mbox{\rm ~for~}}

\def\({\mbox{\rm (}}\def\){\mbox{\rm )}}

\makeatletter

\newcommand{\Rmnum}[1]{\expandafter\@slowromancap\romannumeral #1@}
\makeatother
\newtheorem{theorem}{Theorem}[section]
\newaliascnt{lemma}{theorem}
\newtheorem{lemma}[lemma]{Lemma}
\aliascntresetthe{lemma}

\newaliascnt{proposition}{theorem}

\aliascntresetthe{proposition}

\newaliascnt{fact}{theorem}

\aliascntresetthe{fact}

\newaliascnt{definition}{theorem}
\newtheorem{definition}[definition]{Definition}
\aliascntresetthe{definition}

\newaliascnt{conjecture}{theorem}

\aliascntresetthe{conjecture}

\newaliascnt{corollary}{theorem}
\newtheorem{corollary}[corollary]{Corollary}
\aliascntresetthe{corollary}

\newaliascnt{claim}{theorem}

\aliascntresetthe{claim}

\newaliascnt{problem}{theorem}

\aliascntresetthe{problem}

\newaliascnt{question}{theorem}

\aliascntresetthe{question}

\newaliascnt{remark}{theorem}

\aliascntresetthe{remark}

\newaliascnt{example}{theorem}

\aliascntresetthe{example}

\newaliascnt{notation}{theorem}

\aliascntresetthe{notation}

\begin{document}
\begin{center}
{\Large\bf
Counting Flows of $b$-compatible Graphs}\\[7pt]
\end{center}
\vskip3mm

\begin{center}
Houshan Fu$^{1}$, Xiangyu Ren$^{2,*}$  and Suijie Wang$^{3}$\\[8pt]

 $^{1}$School of Mathematics and Information Science\\
 Guangzhou University\\
Guangzhou 510006, Guangdong, P. R. China\\[12pt]

$^{2}$School of  Mathematics and Statistics\\
 Shanxi University\\
 Taiyuan 030006, Shanxi, P. R. China\\[12pt]

 $^{3}$School of Mathematics\\
 Hunan University\\
 Changsha 410082, Hunan, P. R. China\\[15pt]

$^{*}$Correspondence to be sent to: renxy@sxu.edu.cn \\
Emails: $^{1}$fuhoushan@gzhu.edu.cn, $^{3}$wangsuijie@hnu.edu.cn\\[15pt]
\end{center}
\vskip 3mm
\begin{abstract}
Kochol introduced the assigning polynomial $F(G,\alpha;k)$ to count nowhere-zero $(A,b)$-flows of a graph $G$, where $A$ is a finite Abelian group and $\alpha$ is a $\{0,1\}$-assigning from a family $\Lambda(G)$ of certain nonempty vertex subsets of $G$ to $\{0,1\}$. We introduce the concepts of $b$-compatible graph and $b$-compatible broken bond to give an explicit formula for the assigning polynomials and to examine their coefficients. More specifically, for a function $b:V(G)\to A$, let $\alpha_{G,b}$ be a $\{0,1\}$-assigning of $G$ such  that for each $X\in\Lambda(G)$, $\alpha_{G,b}(X)=0$ if and only if $\sum_{v\in X}b(v)=0$. We show that for any $\{0,1\}$-assigning $\alpha$ of $G$, if there exists a function $b:V(G)\to A$ such that $G$ is $b$-compatible and  $\alpha=\alpha_{G,b}$, then the assigning polynomial $F(G,\alpha;k)$ has the $b$-compatible spanning subgraph expansion
\[
F(G,\alpha;k)=\sum_{\substack{S\subseteq E(G),\\G-S\mbox{ is $b$-compatible}}}(-1)^{|S|}k^{m(G-S)},
\]
and is the following form $F(G,\alpha;k)=\sum_{i=0}^{m(G)}(-1)^ia_i(G,\alpha)k^{m(G)-i}$, where each $a_i(G,\alpha)$ is the number of subsets $S$ of $E(G)$ having $i$ edges such that $G-S$ is $b$-compatible and $S$ contains no $b$-compatible broken bonds with respect to a total order on $E(G)$. Applying the counting interpretation, we also obtain unified comparison relations for the signless coefficients of assigning polynomials. Namely, for any $\{0,1\}$-assignings $\alpha,\alpha'$ of $G$, if there exist functions $b:V(G)\to A$ and $b':V(G)\to A'$ such that $G$ is both $b$-compatible and $b'$-compatible, $\alpha=\alpha_{G,b}$, $\alpha'=\alpha_{G,b'}$ and $\alpha(X)\le\alpha'(X)$ for all $X\in\Lambda(G)$, then
\[
a_i(G,\alpha)\le a_i(G,\alpha') \quad \mbox{ for }\quad i=0,1,\ldots, m(G).
\]
\vspace{1ex}\vspace{1ex}
\end{abstract}
\section{Introduction}\label{Sec-1}
\paragraph{1.1. Background.}
Nowhere-zero $\mathbb{Z}_k$-flows, or modular $k$-flows in graphs, were initially introduced by Tutte in \cite{Tutte1949,Tutte1954} as a dual concept to graph coloring. A planar graph is $k$-colorable if and only if its dual graph allows a nowhere-zero $\mathbb{Z}_k$-flow. An analog to $\mathbb{Z}_k$-flow  is the integer $k$-flow or, simply $k$-flow, where values on edges are integers strictly less than $k$ in absolute value. It is well-known that a graph has a nowhere-zero $k$-flow if and only if it admits a nowhere-zero $A$-flow for any Abelian group $A$ of order $k$ in \cite{Tutte1949}. Nowhere-zero flows are nicely surveyed in \cite{Jaeger1988,Seymour1995}. Tutte also proposed several famous conjectures on nowhere-zero flows, known as the 5-flow conjecture \cite{Tutte1954}, 4-flow conjecture \cite{Tutte1966} and 3-flow conjecture (see Unsolved Problem 48, in \cite{Bondy1976}). These conjectures reveal that nowhere-zero flows are closely related to the edge connectivity of graphs. To address the 3-flow conjecture,  Jaeger et al. \cite{Jaeger1992} proposed the concept of  group connectivity of graphs by introducing a nonhomogeneous form of the nowhere-zero flow in 1992. They said that a graph $G$ is {\em $A$-connected} if there exists an orientation of $G$ such that for any $A$-valued zero-sum function $b\in A^{V(G)}$, $G$ always admits an $A$-valued nowhere-zero function $f\in A^{E(G)}$ whose boundary equals $b$. Such function $f$ is called a {\em nowhere-zero $(A,b)$-flow} in \cite{Lai2006}. Since then, the related topic has attracted a lot of research attentions, such as  \cite{Han-Li-Li-Wang2021,Husek2020,Kral2005,Lai2000,Lai-Mazza2021,Langhede-Thomassen2020,Langhede-Thomassen2024,Li2017}.

Similar to the chromatic number, a fundamental problem is to determine the {\em flow number} of a given graph, that is, the smallest integer $k$ such that the graph admits a nowhere-zero $k$-flow. Counting the number of nowhere-zero flows or proper colorings could be viewed as another aspect of this topic, originating from Birkhoff's attempt to solve the Four-Color problem by means of counting proper colorings of a plane graph  \cite{Birkhoff1912}. These counting problems give rise to two classical polynomials known as the chromatic polynomial and the flow polynomial. The flow polynomial evaluates the number of nowhere-zero $A$-flows, which does not depend on the structure of $A$, but only on the order of $A$. Kochol in \cite{Kochol2002} demonstrated that the number of nowhere-zero $k$-flows is also a polynomial in $k$, distinct from that for nowhere-zero $\mathbb{Z}_k$-flows. A wide field of study of the flow polynomial mainly focuses on its roots\cite{Dong2019}, coefficients\cite{Dong-Koh2007} and computing \cite{Kochol2005,Sekine1997}. However, problems on counting nowhere-zero $(A,b)$-flows had been overlooked until Kochol \cite{Kochol2022} recently showed that there exists a polynomial function counting the  nowhere-zero $(A,b)$-flows of a graph $G$ by an inductive method, called an assigning polynomial. Motivated by Kochol's work, we investigate further properties on the assigning polynomial related to $(A,b)$-flows of a graph. A particular focus will be given to the coefficients of the assigning polynomial, including their combinatorial interpretations  and how they  change as $b$ runs over all locally $A$-valued zero-sum functions.

\paragraph{1.2. Main results.}
Let $G$ be a graph with the vertex set $V(G)$ and the edge set $E(G)$, which may contain multiple edges and loops. Denote by $c(G)$ the number of the components of $G$ and $m(G)=|E(G)|-|V(G)|+c(G)$. Let $A$ be a finite additive Abelian group with the identity $0$. Define $A^{E(G)}$ and $A^{V(G)}$ as the sets of functions from  $E(G)$ to $A$ and $V(G)$ to $A$, respectively. Given an edge orientation of $G$, for each vertex $v\in V(G)$ we denote by $E^+(v)$ the set of edges for which $v$ is the head, and by $E^-(v)$ the set of edges for which $v$ is the tail. The {\em boundary} of a function $f \in A^{E(G)}$ is a function $\partial f\in A^{V(G)}$ defined as follows,
\[
\partial f(v)=\sum_{e \in E^{+}(v)} f(e)-\sum_{e \in E^{-}(v)} f(e),
\]
where the summation refers to the addition in  $A$. For any $b\in A^{V(G)}$,  an {\em $(A,b)$-flow} of $G$ is a function $f \in A^{E(G)}$ whose boundary is $b$. Furthermore, an $(A,b)$-flow $f$ of $G$ is said to be  {\em nowhere-zero} if $f(e)\ne 0$ for any $e\in E(G)$. In particular, when $b=0$, the $(A,0)$-flow of $G$ is precisely the classical group-valued flow in graphs,  abbreviated as an {\em $A$-flow}. We use the notations $F(G,b;A)$ and $F^*(G,b;A)$ to represent the numbers of $(A,b)$-flows and nowhere-zero $(A,b)$-flows of $G$, respectively.

Note from \cite[Proposition 2.1]{Jaeger1992} that for any function $f\in A^{E(G)}$, the boundary $\partial f$ is a {\em locally $A$-valued zero-sum function} on $G$, i.e., the restriction of $\partial f$ to every component of $G$ sums to zero. This implies that the counting function $F^*(G,b;A)=0$ whenever $b$ is not a locally $A$-valued zero-sum function. For this reason, we only need to consider nowhere-zero $(A,b)$-flows of $G$ for locally zero-sum functions $b\in A^{V(G)}$. In this case, the graph $G$ is said to be {\em $b$-compatible}, which is a convenient concept for later use. More formally, if $G$ has a vertex set $V$ and the map $b: V\to A$ is a locally $A$-valued zero-sum function on $G$, then we call $G$ a {\em $b$-compatible graph}.

Denote by $Z(G,A)$  the collection of all locally $A$-valued zero-sum functions $b\in A^{V(G)}$ on $G$.  Our first main result establishes an explicit counting formula of the number of nowhere-zero $(A,b)$-flows for $b\in Z(G,A)$.
\begin{theorem}[Counting Formula]\label{Flow-Polnomial}
For any locally zero-sum function $b\in Z(G,A)$,  we have
\[
F^*(G,b;A)=\sum_{\substack{S\subseteq E(G),\\G-S\mbox{ is $b$-compatible}}}(-1)^{|S|}|A|^{m(G-S)}.
\]
\end{theorem}

For any vertex subset $X\subseteq V(G)$, denote by $G[X]$ and $G-X$ the vertex-induced subgraphs of $G$ on the vertex sets $X$ and $V(G)\setminus X$, respectively. Let $\Lambda(G)$ be the family of nonempty vertex subsets $X\subseteq V(G)$ such that $G[X]$ is connected and $c(G-X)=c(G)$. Adopting the Kochol's notation, a {\em $\{0,1\}$-assigning} of $G$ is a map $\alpha$ from $\Lambda(G)$ to the set $\{0,1\}$. Each function  $b\in A^{V(G)}$ automatically induces a $\{0,1\}$-assigning  $\alpha_{G,b}$ on $\Lambda(G)$ satisfying that for each $X\in\Lambda(G)$,
\[
\alpha_{G,b}(X)=\begin{cases}
0,& \mbox{ if } \sum_{v\in X}b(v)=0;\\
1,& \mbox{ otherwise }.
\end{cases}
\]

Kochol's main result \cite[Theorem 1]{Kochol2022} showed that there exists a polynomial function counting nowhere-zero $(A,b)$-flows.
\begin{theorem}[\cite{Kochol2022}, Theorem 1]\label{Polynomial-Theorem}
Let $b\in Z(G,A)$. There exists a polynomial function $F(G,\alpha;k)$ $(\alpha=\alpha_{G,b})$ of $k$ such that $F(G,\alpha;k)=F^*(G,b;A')$ for any Abelian group $A'$ of  order $k$ and each $b'\in Z(G,A')$ satisfying $\alpha_{G,b'}=\alpha$. Moreover, $F(G,\alpha; k)$ has degree $m(G)$ if $F^*(G,b;A)\ne 0$.
\end{theorem}

The polynomial function $F(G,\alpha;k)$ in \autoref{Polynomial-Theorem} is referred to as an {\em $\alpha$-assigning polynomial} of $G$, or simply {\em assigning polynomial}. Motivated by \autoref{Flow-Polnomial}, below we provide an explicit expression for the $\alpha$-assigning polynomial in terms of  $b$-compatible spanning subgraph expansion.
\begin{theorem}\label{Assigningpolynomial-Express}
For any $\{0,1\}$-assigning $\alpha$ of $G$, if there exist an Abelian group $A$ and a locally zero-sum function $b\in Z(G,A)$ such that $\alpha=\alpha_{G,b}$, then
\[
F(G,\alpha,k)=\sum_{\substack{S\subseteq E(G),\\G-S\mbox{ is $b$-compatible}}}(-1)^{|S|}k^{m(G-S)}.
\]
\end{theorem}

The natural question is whether there exists a combinatorial interpretation for the coefficients of assigning polynomials. To address this, we introduce the concept of the $b$-compatible broken bond. A bond of a graph is defined as a minimal nonempty edge cut, which is the smallest nonempty subset of edges whose removal from the graph results in an increase in the number of components.
\begin{definition}
{\rm Let $G$ be a graph with a total order  $\prec$ on $E(G)$. For any $b\in Z(G,A)$, a subset $S\subseteq E(G)$ is said to be:
\begin{itemize}
\item a {\em $b$-compatible bond} of $G$ if $S$ is a bond of $G$ and $G-S$ is $b$-compatible;
\item a {\em $b$-compatible broken bond} of $G$ if $S$ is obtained from a $b$-compatible bond by removing the maximal element with respect to the total order $\prec$.
\end{itemize}
}
\end{definition}
The following result shows that if $\alpha=\alpha_{G,b}$ for some $b\in Z(G,A)$, then the coefficients of the $\alpha$-assigning polynomial $F(G,\alpha;k)$ can be further interpreted in terms of $b$-compatible spanning subgraphs that do not contain $b$-compatible broken bonds. In addition, \autoref{NBB-Theorem} also provides a new perspective on addressing certain group connectivity problems in graphs. A detailed discussion will appear at the end of \autoref{Sec-3}.

\begin{theorem}[Generalized Broken Bond Theorem]\label{NBB-Theorem}
Let $G$ be a graph with a total order $\prec$ on $E(G)$, and $\alpha$ be a $\{0,1\}$-assigning of $G$. Write $F(G,\alpha;k)=\sum_{i=0}^{m(G)}(-1)^ia_i(G,\alpha)k^{m(G)-i}$.  If there exist an Abelian group $A$ and a locally zero-sum function $b\in Z(G,A)$ such that $\alpha=\alpha_{G,b}$, then each $a_i(G,\alpha)$ is the number of subsets $S$ of $E(G)$ having $i$ edges such that $G-S$ is $b$-compatible and $S$ contains no $b$-compatible broken bonds with respect to $\prec$.
\end{theorem}

Our last main result will establish a unified order-preserving relation from $\{0,1\}$-assignings to assigning polynomials when both are naturally ordered.
\begin{theorem}[Comparison of Coefficients]\label{Comparison}
Suppose $\{0,1\}$-assignings $\alpha,\alpha'$ of $G$ satisfy $\alpha=\alpha_{G,b}$ and $\alpha'=\alpha_{G,b'}$ for locally zero-sum functions $b\in Z(G,A)$ and $b'\in Z(G,A')$. If $\alpha(X)\le \alpha'(X)$ for all $X\in \Lambda(G)$, then the signless coefficients of the assigning polynomials $F(G,\alpha;k)$ and $F(G,\alpha';k)$ satisfy 
\[
a_i(G,\alpha)\le a_i(G,\alpha') \quad \mbox{ for }\quad i=0,1,\ldots,m(G).
\]
In particular, $a_i(G,0)\le a_i(G,\alpha_{G,b})$ for all $b\in Z(G,A)$.
\end{theorem}
\paragraph*{1.3. Organization.}
\autoref{Sec-2} is devoted to showing \autoref{Flow-Polnomial} and \autoref{Assigningpolynomial-Express}. In \autoref{Sec-3}, we focus on the proof of \autoref{NBB-Theorem} and its application to group connectivity problems in groups. Finally, we verify \autoref{Comparison} and provide the decomposition formulas for the number of (nowhere-zero) edge-valued functions using two different counting methods in \autoref{Sec-4}.
\section{Counting formula}\label{Sec-2}
In this section, we will prove \autoref{Flow-Polnomial} (Counting Formula) and \autoref{Assigningpolynomial-Express}. Let us review some necessary notations and definitions on graphs. For a graph $G$, a graph $G'$ is a subgraph of $G$ if $V(G')\subseteq V(G)$ and $E(G')\subseteq E(G)$. In particular, a subgraph $G'$ of $G$ is called a {\em spanning subgraph} if $V(G')=V(G)$. We say that $G'$ is the {\em vertex-induced subgraph} spanned by $V(G')\subseteq V(G)$ if $G'$ contains all edges of $G$ that join two vertices in $V(G')$, denoted by $G[V(G')]$.  For an edge subset $E'\subseteq E(G)$, $G-E'$ is the subgraph obtained by removing the edges in $E'$. Conversely, suppose $G'$ is a proper subgraph containing no edge $e\in E(G)$, we denote by $G'+e$ the subgraph of $G$ by adding the edge $e$.

An edge subset $F\subseteq E(G)$  is an {\em edge cut} in $G$ if there exists a partition \{X,Y\} of $V(G)$ such that $F=E[X,Y]$, where $E[X,Y]$ is the  set of edges of $G$ with one end in $X$ and the other end in $Y$. A minimal nonempty edge cut in $G$ is called a {\em bond}. One should observe that a bond $F$ of $G$ corresponds precisely to a unique component  of $G$ containing $F$.

Below we will prove the first main result using the M\"obius Inversion Formula in \cite[Theorem 2.25]{Bondy2008}. The M\"obius Inversion Formula states that if the function $g: 2^{E(G)}\to A$ is defined as
\[
g(S):=\sum_{S\subseteq X\subseteq E(G)}f(X) \quad\mbox{ for }\quad \forall\,S\subseteq E(G),
\]
then, for all $S\subseteq E(G)$,
\begin{equation}\label{Inversion}
f(S)=\sum_{S\subseteq X\subseteq E(G)}(-1)^{|X|-|S|}g(X),
\end{equation}
where $f$ is a function from the family $2^{E(G)}$ of all edge subsets of $E(G)$ to $A$.
\begin{proof}[Proof of \autoref{Flow-Polnomial}]
For any $S\subseteq E(G)$, let $F_S(G,b;A)$ be the number of $(A,b)$-flows $f$ of $G$ satisfying $f(e)=0$ for all $e\in S$ , and $F^*_S(G,b;A)$ be the number of $(A,b)$-flows $f$ of $G$ such that $f(e)=0$ for all $e\in S$ and $f(e)\ne 0$ for $e\in E(G)\setminus S$. Then
\[
F_{S}(G,b;A)=F(G-S,b;A)\quad\And\quad F^*_\emptyset(G,b;A)=F^*(G,b;A).
\]
Immediately, we have
\[
F(G,b;A)=\sum_{S\subseteq E(G)}F^*_S(G,b;A).
\]
It follows from \eqref{Inversion} that
\begin{equation}\label{Inclusion-Exclusion}
F^*(G,b;A)=\sum_{S\subseteq E(G)}(-1)^{|S|}F_S(G,b;A)=\sum_{S\subseteq E(G)}(-1)^{|S|}F(G-S,b;A).
\end{equation}
Since $b$ is a locally $A$-valued zero-sum function on $G$, we have $F(G,b;A)\ne 0$. For a fixed $(A,b)$-flow $f$ of $G$, every $(A,b)$-flow $f'$ of $G$ can give rise to a unique $A$-flow $f'-f$ of $G$. Conversely, every $A$-flow $f''$ of $G$ also gives rise to a unique $(A,b)$-flow $f+f''$ of $G$. It follows from the classical flow theory that
\begin{equation}\label{Eq0}
F(G,b;A)=F(G,0;A)=|A|^{m(G)}.
\end{equation}
Notice that the corresponding boundary function $\partial f$ of each function $f\in A^{E(G)}$ is a locally $A$-valued zero-sum function. So, $F(G,b;A)=0$ if $G$ is $b$-incompatible. Together with \eqref{Eq0}, we obtain that $F^*(G,b;A)$ in \eqref{Inclusion-Exclusion} can be simplified to the following form
\[
F^*(G,b;A)=\sum_{\substack{S\subseteq E(G),\\G-S\mbox{ is $b$-compatible}}}(-1)^{|S|}|A|^{m(G-S)}.
\]
This finishes the proof.
\end{proof}

Next, we proceed to verify \autoref{Assigningpolynomial-Express}.
\begin{proof}[Proof of \autoref{Assigningpolynomial-Express}]
For the fixed $b$, define the polynomial $P(G,b;k)$ in $k$ as follows:
\[
P(G,b;k)=\sum_{\substack{S\subseteq E(G),\\G-S\mbox{ is $b$-compatible}}}(-1)^{|S|}k^{m(G-S)},
\]
where the degree of $P(G,b;k)$ is at most $m(G)$. Consider the Abelian group $A^i=A\times\cdots\times A$ consisting of all $i$-tuples $(a_1,\ldots,a_i)$  with each component $a_j\in A$, for $i=1,2,\ldots, m(G)+1$. Choose an element $b_j=(b_{j1},\ldots,b_{ji})$ in $A^i$ such that the first component $b_{j1}=b_{11}$ and all remaining components equal $0$ for each $i$. Note $b_1=b$. Then, for any $S\subseteq E(G)$ and every $i=1,\ldots,m(G)+1$, $G-S$ is $b$-compatible if and only if $G-S$ is $b_i$-compatible. Together with \autoref{Flow-Polnomial}, for $i=1,2,\ldots, m(G)+1$, we arrive at
\[
F^*(G,b_i;A^i)=\sum_{\substack{S\subseteq E(G),\\G-S\mbox{ is $b$-compatible}}}(-1)^{|S|}|A^i|^{m(G-S)}
\]
It follows that $P(G,b;|A^i|)=F^*(G,b_i;A^i)$ for every $i$. Noticing $\alpha_{G,b_i}=\alpha$ for each $i$, we derive $F(G,\alpha;|A^i|)=F^*(G,b_i;A^i)$ via \autoref{Polynomial-Theorem}. Therefore,  for $i=1,\ldots, m(G)+1$, we have
\[
F(G,\alpha;|A^i|)=P(G,b;|A^i|).
\]
The fundamental theorem of algebra says that if two polynomials of the degrees at most $n$ agree on $n+1$ distinct values, then they are the same polynomial. Since the orders of $A^1,\ldots, A^{m(G)+1}$ are distinct, and both $F(G,\alpha;k)$ and $P(G,b;k)$ are the polynomials in $k$ of degree at most $m(G)$,  the fundamental theorem of algebra implies $F(G,\alpha;k)=P(G,b;k)$. Namely, $F(G,\alpha;k)$ can be expressed as
\[
F(G,\alpha;k)=\sum_{\substack{S\subseteq E(G),\\G-S\mbox{ is $b$-compatible}}}(-1)^{|S|}k^{m(G-S)},
\]
which completes the proof.
\end{proof}

\section{Generalization of Broken Bond Theorem} \label{Sec-3}
\autoref{NBB-Theorem} provides a counting interpretation of the signless coefficients of the assigning polynomials in terms of $b$-compatible broken bonds. In this section, we focus on showing \autoref{NBB-Theorem} (Generalized Broken Bond Theorem) and exploring its applications to group connectivity problems of graphs. To this end, the following definitions and lemmas are required. A partition $\pi$ of a finite set $S$ is a collection $\{B_1,\ldots,B_k\}$ of nonempty subsets of $S$ such that
\[
\bigsqcup_{i=1}^{k}B_i=S\quad\And\quad B_i\cap B_j=\emptyset\quad\For\quad  i\ne j\in [k],
\]
where each $B_i$ is called a {\em block} of $\pi$. Associated with a graph $G$, a partition $\pi=\{B_1,\ldots,B_k\}$ of $V(G)$ is referred to as the {\em connected partition} if each subgraph $G[B_i]$ spanned by the block $B_i$ is connected.  In this sense, each $B_i$ of  $V(G)$ is said to be a {\em connected block}. In particular, all the components $G_1,\ldots,G_k$ of $G$ automatically induces a connected partition $\big\{V(G_1),\ldots,V(G_k)\big\}$ of $V(G)$. Then we observe the next lemmas.
\begin{lemma}\label{Inclusion}
Let $S_1\subseteq S_2\subseteq E(G)$ and $b\in A^{V(G)}$. If $G-S_2$ is $b$-compatible, then $G-S_1$ is also $b$-compatible.
\end{lemma}
\begin{proof}
Since $S_1\subseteq S_2$,  each connected block corresponding to a component of $G-S_1$ can be expressed as the disjoint union of some connected blocks corresponding to the components of $G-S_2$. Together with that $G-S_2$ is $b$-compatible, we can deduce that $G-S_1$ is also $b$-compatible.
\end{proof}

\begin{lemma}\label{No-Broken-Bond}
Let $G$ be a graph with a total order $\prec$ on $E(G)$ and $S\subseteq E(G)$ contain a $b$-compatible broken bond $B$ for some $b\in Z(G,A)$.
Suppose $B=F\setminus e_F$ for some $b$-compatible bond $F$ of $G$, where $e_F$ is the maximal member in $F$ with respect to the total order $\prec$. If $G-S$ is $b$-compatible and $e_F\notin S$, then $G-S-e_F$ is also $b$-compatible.
\end{lemma}
\begin{proof}
Suppose $e_F$ is incident with vertices $u$ and $v$. Let $G_{uv}, G_u$ and $G_v$ be the connected components of $G-S$ containing the edge $e_F=uv$,  of $G-S-e_F$ containing the vertex $u$ and of $G-S-e_F$ containing the vertex $v$, respectively. Since $B$ is a broken bond and $F=B\sqcup e_F$ is a bond of $G$, $G_{uv}-e_F$ is partitioned into two connected components $G_u$ and $G_v$ in $G-S-e_F$. This implies that the connected partition $\pi_{G-S-e_F}$ of $V(G)$ induced by the components of $G-S-e_F$ is obtained from the connected partition $\pi_{G-S}$ of $V(G)$ corresponding to the components of $G-S$ by partitioned the connected block $V(G_{uv})$ of $G-S$ into two extra connected blocks $V(G_u)$ and $V(G_v)$ of $G-S-e_F$. Hence, to prove that $G-S-e_F$ is $b$-compatible, it is enough to show
\begin{equation}\label{Eq1}
\sum_{w\in V(G_u)}b(w)=0\quad\And\quad \sum_{w\in V(G_v)}b(w)=0.
\end{equation}
Let $G_F$ be the unique component of $G$ containing $F$. As $F$ is a bond of $G$, we can assume that $G_F-F$ is partitioned into two components $G_{F,u}$ containing $G_u$ as a subgraph of $G_F$ and $G_{F,v}$ containing $G_v$ as a subgraph of $G_F$. Noting that $G$ and $G-F$ are $b$-compatible, we have
\begin{equation}\label{Eq4}
\sum_{w\in V(G_{F,u})}b(w)=0\quad\And\quad \sum_{w\in V(G_{F,v})}b(w)=0.
\end{equation}
Let $\pi_{G_{F,u}}=\big\{B_1,\ldots,B_k,V(G_u)\big\}$ be a partition of $V(G_{F,u})$, where each $B_i$ is the connected block induced by the corresponding component of $G-S$. Since $G-S$ is $b$-compatible, for each block $B_i$ we have $\sum_{w\in B_i}b(w)=0$. Together with the first equation in \eqref{Eq4}, we acquire
\[
\sum_{w\in V(G_u)}b(w)=\sum_{w\in V(G_{F,u})}b(w)-\sum_{i=1}^k\sum_{w\in B_i}b(w)=0.
\]
Likewise, we can deduce $\sum_{w\in V(G_v)}b(w)=0$. So, we have shown \eqref{Eq1} holds. Namely, $G-S-e_F$ is $b$-compatible.
\end{proof}

\begin{lemma}\label{NBB-Theorem0}
Let $G$ be a graph with a total order $\prec$ on $E(G)$. For any $b\in Z(G,A)$, the counting function $F^*(G,b;A)$ can be written as $F^*(G,b;A)=\sum_{i=0}^{m(G)}(-1)^ia_i(G,b)|A|^{m(G)-i}$, where each $a_i(G,b)$ is the number of subsets $S$ of $E(G)$ having $i$ edges such that $G-S$ is $b$-compatible and $S$ contains no $b$-compatible broken bonds with respect to $\prec$.
\end{lemma}
\begin{proof}
Suppose the $b$-compatible graph $G$ has $q\ge 0$ $b$-compatible broken bonds. For every $b$-compatible broken bond $B$, let $e_B$ be the maximal member in $B$ with respect to the total order $\prec$. We can arrange the $b$-compatible broken bonds $B_1,\ldots, B_q$ such that $e_{B_1}\preceq e_{B_2}\cdots\preceq  e_{B_q}$. Let $\mathcal{S}_1$ be the collection of edge subsets $S$ of $G$ containing the $b$-compatible broken bond $B_1$ such that  each corresponding spanning subgraph $G-S$ is $b$-compatible. For $i=2,\ldots,q$, let $\mathcal{S}_i$ be the collection of edge subsets $S$ of $G$  containing the $b$-compatible broken bond $B_i$ but not any of $B_1,\ldots, B_{i-1}$ such that  each corresponding spanning subgraph $G-S$ is $b$-compatible. In addition, let $\mathcal{S}_{q+1}$ be the collection of edge subsets $S$ of $G$ containing no $b$-compatible broken bonds such that each corresponding spanning subgraph $G-S$  is $b$-compatible. Then $\{\mathcal{S}_1,\mathcal{S}_2,\ldots,\mathcal{S}_{q+1}\}$ induces a partition of the set of $b$-compatible spanning subgraphs of $G$. Together with \autoref{Flow-Polnomial}, then the counting function $F^*(G,b;A)$ can be written in the following form
\begin{align}\label{Eq2}
F^*(G,b;A)&=\sum_{\substack{S\subseteq E(G),\\G-S\mbox{ is $b$-compatible}}}(-1)^{|S|}F(G-S,b;A)\nonumber\\
&=\sum_{i=1}^{q+1}\sum_{S\in\mathcal{S}_i}(-1)^{|S|}F\big(G-S,b;A\big).
\end{align}

We first consider $\mathcal{S}_1$. Assume $B_1=F_1\setminus e_{F_1}$ for some $b$-compatible bond $F_1$, where $e_{F_1}$ is the maximal member in $F_1$ with respect to the total order $\prec$. By \autoref{Inclusion}, each member $S$ in $\mathcal{S}_1$ containing $e_{F_1}$  gives rise to a unique member $S\setminus e_{F_1}$ in $\mathcal{S}_1$ that does not contain $e_{F_1}$. Conversely, according to \autoref{No-Broken-Bond}, each member $S$ in $\mathcal{S}_1$ containing no $e_{F_1}$ gives rise to a unique member $S\sqcup e_{F_1}$ in $\mathcal{S}_1$ containing $e_{F_1}$. In addition, since $B_1\sqcup e_{F_1}$ is a bond of $G$, for each $S\in\mathcal{S}_1$ containing $e_{F_1}$, we have $c(G-S)=c(G-S+e_{F_1})+1$. Thus, from \eqref{Eq0}, we obtain
\[
\sum_{S\in\mathcal{S}_1, e_{F_1}\in S}\big((-1)^{|S|}F(G-S,b;A)+(-1)^{|S|-1}F(G-S+e_{F_1},b;A)\big)=0.
\]
Namely, we arrive at $\sum_{S\in\mathcal{S}_1}(-1)^{|S|}F\big(G-S,b;A\big)=0$ in \eqref{Eq2}.

Consider $\mathcal{S}_2$. Assume $B_2=F_2\setminus e_{F_2}$ for some $b$-compatible bond $F_2$, where $e_{F_2}$ is the maximal member in $F_2$ with respect to the total order $\prec$. It follows from the ordering $e_{B_1}\preceq e_{B_2}$ that  $e_{B_1}\preceq e_{B_2}\prec e_{F_2}$. So $e_{F_2}\notin B_1\cup B_2$. Combining this, using the same argument as the case $\mathcal{S}_1$, we can also verify that every element $S$ in $\mathcal{S}_2$ containing no $e_{F_2}$ gives rise to a unique member $S\sqcup e_{F_2}$ in $\mathcal{S}_2$ containing $e_{F_2}$, and vice versa. Then we also have $\sum_{S\in\mathcal{S}_2}(-1)^{|S|}F\big(G-S,b;A\big)=0$ in \eqref{Eq2}.

Likewise, we can deduce the total contribution of the members in $\mathcal{S}_i$ to $F^*(G,b;A)$ is zero for each $i=1,2,\ldots,q$. Therefore, we have
\begin{align}\label{Eq3}
F^*(G,b;A)&=\sum_{S\in\mathcal{S}_{q+1}}(-1)^{|S|}F(G-S,b;A)\nonumber\\
&=\sum_{i=0}^{|E(G)|}\sum_{S\in\mathcal{S}_{q+1},\,|S|=i}(-1)^iF\big(G-S,b;A\big).
\end{align}
As each $S\in\mathcal{S}_{q+1}$ contains no $b$-compatible broken bonds and $G-S$ is $b$-compatible, we can derive from \autoref{Inclusion} that $S$ contains no bonds. This implies
\[
c(G-S)=c(G),\quad S\in\mathcal{S}_{q+1}.
\]
Combining $F(G-S,b;A)=|A|^{m(G-S)}$ in \eqref{Eq0} and \eqref{Eq3}, we have
\[
F^*(G,b;A)=\sum_{i=0}^{m(G)}(-1)^ia_i(G,b)|A|^{m(G)-i},
\]
which completes the proof.
\end{proof}
Now, we turn to the proof of our third main result.
\begin{proof}[Proof of \autoref{NBB-Theorem}]
For the fixed $b$, define the polynomial $P(G,b;k)$ in $k$ as follows:
\[
P(G,b;k)=\sum_{i=0}^{m(G)}(-1)^ia_i(G,b)k^{m(G)-i},
\]
where $a_i(G,b)$ is the number of subsets $S$ of $E(G)$ having $i$ edges such that $G-S$ is $b$-compatible and $S$ contains no $b$-compatible broken bonds with respect to $\prec$. Analogous to the proof of \autoref{Assigningpolynomial-Express}, we can obtain
\[
F(G,\alpha;k)=P(G,b;k).
\]
This implies $a_i(G,\alpha)=a_i(G,b)$ for each $i=0,1,\ldots,m(G)$, which finishes the proof.
\end{proof}

It is worth noting  that the coefficients $a_i(G,\alpha)$ in \autoref{NBB-Theorem} are independent of the choice of edge order and of $b$ whenever $\alpha_{G,b}=\alpha$. When $b$ is the zero function $0$, then $\alpha_{G,0}=0$ and $F(G,0;k)$ is precisely the classical flow polynomial. In this case, the $0$-compatible bond and $0$-compatible broken bond correspond to the bond and broken bond of $G$, respectively. Therefore, \autoref{NBB-Theorem} directly leads to a counting explanation of the coefficients for  the classical flow polynomials. This can be viewed as the dual result of Whitney's celebrated Broken Circuit Theorem \cite{Whitney1932},  since the bonds of  a planar graph are precisely the circuits of its dual graph.
\begin{corollary}[Broken Bond Theorem]
Let $G$ be a graph with a total order $\prec$ on $E(G)$. Write $F(G,0;k)=\sum_{i=0}^{m(G)}(-1)^ia_i(G,0)k^{m(G)-i}$. Then each signless coefficient $a_i(G,0)$ is the number of  spanning subgraphs of $G$ having $i$ edges and containing no broken bonds with respect to $\prec$.
\end{corollary}

It is a pleasant surprise that the Generalized Broken Bond Theorem provides a new perspective on addressing certain group connectivity problems in graphs. Notably, a well-known property in \cite{Tutte1949} (see also \cite{Zhang1997}) says that if a graph has a nowhere-zero $A$-flow, then it admits a nowhere-zero $A'$-flow for any Abelian group $A'$ whose order is at least the order of $A$. This does not extend to the more general concept of group connectivity observed by Jaeger et al. \cite{Jaeger1992}. Recently,  Langhede and Thomassen in \cite{Langhede-Thomassen2020} confirmed that the property holds if $A'$ is sufficiently large compared to $A$. They also demonstrated that $G$ is always $A$-connected whenever $A$ has a sufficiently large order, which can be explained by \autoref{NBB-Theorem}. More specifically, notice that if a graph $G$ is $A$-connected, then $G$ is 2-edge-connected. Since for any 2-edge-connected graph $G$, the empty set does not contain $b$-compatible broken bonds, the coefficient $a_0(G,\alpha_{G,b})=1$ in the assignning polynomial $F(G,\alpha_{G,b};k)$ follows from \autoref{NBB-Theorem}. Thus, when $G$ is 2-edge-connected and the Abelian group $A$ has a large enough order,  $F^*(G,b;A)=F(G,\alpha_{G,b};|A|)>0$ always holds for each $A$-valued zero-sum function $b\in A^{V(G)}$. This further implies that a 2-edge-connected graph $G$ is always $A$-connected whenever $A$ has a sufficiently large order.

Furthermore, note the basic fact that for a 2-edge-connected graph $G$, $G$ is $A$-connected for some Abelian group $A$ if and only if $F(G,\alpha_{G,b};|A|)>0$ for any $b\in Z(G,A)$. As a consequence, the following corollary is easily observed.
\begin{corollary}\label{Group-Connectivity}
Let $A,A'$ be two Abelian groups of the same order and $G$ an $A'$-connected graph. If  each $\{0,1\}$-assigning $\alpha_{G,b}$ with $b\in Z(G,A)$ is an $\{0,1\}$-assigning $\alpha_{G,b'}$ for some $b'\in Z(G,A')$, then $G$ is $A$-connected.
\end{corollary}

\section{Comparison of coefficients}\label{Sec-4}
\autoref{Comparison} characterizes the change in the signless coefficients of the assigning polynomial $F(G,\alpha;k)$ ($\alpha=\alpha_{G,b}$) as $b$ runs over all locally $A$-valued zero-sum function. In this section, we will verify \autoref{Comparison} via \autoref{NBB-Theorem} and handle the decomposition of the number of (nowhere-zero) edge-valued functions using two different counting methods. Let $G$ be a $b$-compatible graph having a total order $\prec$ on $E(G)$. For convenience, we denote the sets of  $b$-compatible bonds and $b$-compatible broken bonds of $G$ by $\mathcal{B}(G,b)$ and $\mathcal{B}'(G,b)$, respectively. Let $\mathcal{B}''(G,b)$ be the family of edge subsets $S\subseteq E(G)$ such that $G-S$ is $b$-compatible and $S$ does not contain $b$-compatible broken bonds with respect to $\prec$.
\begin{proof}[Proof of \autoref{Comparison}]
First note from \cite[Theorem 2.15]{Bondy2008} that for a subset $X\subseteq V(G)$, if $G$ is connected, then the nonempty edge cut $E[X,V(G)\setminus X]$ is a bond of $G$ if and only if the vertex-induced subgraphs $G[X]$ and $G[V(G)\setminus X]$ are connected. Therefore, for any $X\in\Lambda(G)$, the edge subset $E[X,V(G)\setminus X]$ is a bond of $G$. $\mathcal{B}(G,b)\supseteq \mathcal{B}(G,b')$ follows from that $\alpha_{G,b}(X)\le\alpha_{G,b'}(X)$ for all $X\in\Lambda(G)$. This further implies $\mathcal{B}'(G,b)\supseteq \mathcal{B}'(G,b')$. To obtain  $a_i(G,\alpha)\le a_i(G,\alpha')$, from \autoref{NBB-Theorem}, it is enough to verify $\mathcal{B}''(G,b)\subseteq \mathcal{B}''(G,b')$. Given a member $S\in \mathcal{B}''(G,b)$, we have that $S$ does not contain $b'$-compatible broken bonds according to $\mathcal{B}'(G,b)\supseteq \mathcal{B}'(G,b')$. This means that $S$ contains no $b'$-compatible bonds. Applying \autoref{Inclusion}, we deduce that $S$ contains no bonds of $G$. This implies that the connected partition of $V(G)$ induced by the components of $G-S$ agrees with the connected partition of $V(G)$ induced by the components of $G$. Since $G$  is $b'$-compatible, $G-S$ is also $b'$-compatible. Namely, $S\in \mathcal{B}''(G,b')$, which completes the proof.
\end{proof}

Furthermore, from \autoref{NBB-Theorem}, we can easily see that if $G$ is a bridgeless graph and the $\{0,1\}$-assigning $\alpha$ of $G$ equals $0$, then each signless $a_i(G,0)$ of the classical flow polynomial $F(G,0;k)$ is a positive integer. Applying this to the relation $a_i(G,0)\le a_i(G,\alpha_{G,b})$ for any $b\in Z(G,A)$ in \autoref{Comparison}, we can directly derive the positivity of coefficients for the general assigning polynomial $F(G,\alpha;k)$ of the bridgeless graph $G$.
\begin{corollary}
For any $b\in Z(G,A)$, if $G$ is a bridgeless graph, then each signless coefficient $a_i(G,\alpha)$ of the assigning polynomial $F(G,\alpha;k)$ ($\alpha=\alpha_{G,b}$) is a positive integer.
\end{corollary}

At the end of this section, we further provide the simple decomposition formulas for the numbers of nowhere-zero $A$-valued edge functions and $A$-valued edge functions in terms of the counting functions $F^*(G,b;A)$ and $F(G,b;A)$ as $b$ runs over members in $Z(G,A)$, respectively.
\begin{theorem}[Decomposition Formula] Let $G$ be a graph with $m$ edges. We have
\[
\sum_{b\in Z(G,A)}F^*(G,b;A)=\big(|A|-1\big)^m\quad\And\quad \sum_{b\in Z(G,A)}F(G,b;A)=|A|^m.
\]
\end{theorem}
\begin{proof}
We will prove the former equation by counting the number of pairs $(b,f)$ such that $b\in Z(G,A)$ and $f$ is a nowhere-zero $(A,b)$-flow of $G$ in two different ways. It is clear that  $\sum_{b\in Z(G,A)}F^*(G,b;A)$ is the sum of the numbers of nowhere-zero $(A,b)$-flows as $b$ ranges over all members in $Z(G,A)$. On the other hand, if we consider a function $f$ from $A\setminus\{0\}$ to $E(G)$, then the boundary $\partial f$  is automatically a locally $A$-valued zero-sum function on $G$. In this case, $G$ is $\partial f$-compatible. Thus, the number of pairs $(f,\partial f)$ as $f$ runs over elements in $\big(A\setminus\{0\}\big)^{E(G)}$ equals $\big(|A|-1\big)^m$. Therefore, we obtain the equation $\sum_{b\in Z(G,A)}F^*(G,b;A)=\big(|A|-1\big)^m$.
Using a similar argument as the former equation, we can deduce the latter equation, which completes the proof.
\end{proof}
\section*{Acknowledgements}
The first author is supported by National Natural Science Foundation of China under Grant No. 12301424. The corresponding author is supported by Natural Science Foundation of Shanxi Province under Grant No. 202203021212484.

\end{document}